\documentclass[number,citesort,seceqn,dvips]{arxbj}
\usepackage{mathbh}

\aid{0}
\volume{17}
\issue{4}
\pubyear{2011}
\firstpage{1248}
\lastpage{1267}
\doi{10.3150/10-BEJ310}

\makeatletter
\newcommand{\eqref}[1]{(\ref{#1})}
\newcommand{\sign}{\operatorname{sign}}
\newcommand{\1}{\mathbh{1}}

\newremark{rk}{Remark}[section]
\newtheorem{lemma}[rk]{Lemma}
\newtheorem{propo}[rk]{Proposition}
\newremark{defn}[rk]{Definition}
\newtheorem{theo}[rk]{Theorem}
\newtheorem{coro}[rk]{Corollary}

\makeatother

\begin{document}
\begin{frontmatter}

\title{Some particular self-interacting diffusions: Ergodic behaviour
and almost sure convergence}
\runtitle{Self-interacting diffusions: Ergodicity and convergence}

\begin{aug}
\author[1]{\fnms{S\'{e}bastien} \snm{Chambeu}\thanksref{1}\ead[label=e1]{Sebastien.Chambeu@math.u-psud.fr}} \and
\author[2]{\fnms{Aline} \snm{Kurtzmann}\corref{}\thanksref{2}\ead[label=e2]{aline.kurtzmann@iecn.u-nancy.fr}}
\runauthor{S. Chambeu and A. Kurtzmann}
\address[1]{Laboratoire de mod\'{e}lisation
stochastique et statistique, Universit\'{e} Paris Sud, B\^{a}timent 425,
F-91405 Orsay Cedex, France. \printead{e1}}
\address[2]{Institut Elie Cartan, Universit\'{e} Nancy 1, B.P. 70239, F-54506 Vand\oe uvre-l\`{e}s-Nancy Cedex, France.
\printead{e2}}
\end{aug}

\received{\smonth{11} \syear{2008}}
\revised{\smonth{6} \syear{2007}}

\begin{abstract}
This paper deals with some self-interacting diffusions
$(X_t,t\geq 0)$ living on $\mathbb{R}^d$. These diffusions are
solutions to stochastic differential equations:
\[\mathrm{d}X_t =
\mathrm{d}B_t - g(t)\nabla V(X_t - \overline{\mu}_t)\,\mathrm{d}t,
\]
where $\overline{\mu}_t$ is the
empirical mean of the process $X$, $V$ is an asymptotically
strictly convex potential and $g$ is a given function. We study
the ergodic behaviour of $X$ and prove that it is strongly related
to $g$. Actually, we show that $X$ is ergodic (in the limit quotient sense) if and only if $\overline{\mu}_t$
converges a.s. We also give some conditions (on $g$ and $V$) for the almost sure
convergence of $X$.
\end{abstract}


\end{frontmatter}

\section{Introduction}
Processes with path interaction have been an intensive research area
since the seminal work of Norris, Rogers and Williams \cite{NoRoWi}.
More precisely, self-interacting diffusions have been first introduced
by Durrett and Rogers \cite{DuRo} under the name of Brownian polymers.
They proposed a~model for the shape of a growing polymer. Denoting by
$X_t$ the location of the end of the polymer at time $t$, $X$ satisfies
a stochastic differential equation (SDE) with a drift term depending on
its own occupation measure (in dimension 1, we define it through the
local time of $X$). One is then interested in rescaling $X$ (see
\cite{CLeJ,CrMo,HeRo,MoTa,Rai97}). Another model of polymers has been
proposed by Bena\"{i}m, Ledoux and Raimond \cite {BeLeRa}. They have
studied a class of self-interacting diffusions depending on the
empirical measure. When the process is living on a compact Riemannian
manifold, they have proved that the asymptotic behaviour of the
empirical measure can be related to the analysis of some deterministic
dynamical flow defined on the space of the Borel probability measures.
Bena\"{i}m and Raimond \cite{BeRa05} went further in this study and in
particular, they gave sufficient conditions for the a.s. convergence of
the empirical measure. Very recently, Raimond \cite{Rai06} has
generalized the previous work. He has studied the asymptotic properties
of a process $X$, living on a Riemannian compact manifold $M$, solution
to the SDE
%
\begin{equation}\label{eq:raimond}
\mathrm{d}X_t = \mathrm{d}B_t - g(t)\nabla V*\mu_t (X_t)\,\mathrm{d}t,
\end{equation}
with $V*\mu_t(x) = \frac{1}{t}\int_0^t V(x,X_s)\,\mathrm{d}s$, $\mu_t =
\frac{1}{t}\int_0^t \delta_{X_s}\,\mathrm{d}s$, $\vert g(t)\vert \leq
a\log(t)$ and $g'(t) = \mathrm{O}(t^{-\gamma})$ with $0<\gamma \leq 1$. He has
proved that, unless $g$ is constant, the approximation of $\mu_t$ by a
deterministic flow is no longer valid. He has more particularly
investigated the example $M = \mathbb{S}^n$ and $V(x,y) = -\cos d(x,y)$
(where $d$ is the geodesic distance on $\mathbb{S}^n$) and proved that
a.s. $\mu_t$ converges weakly toward a Dirac measure. For an overview
on reinforced processes, we refer the reader to Pemantle's survey
\cite{Pemantle}.

In the present paper, we are concerned with some self-interacting
processes living on $\mathbb{R}^d$. Consider a smooth potential $V\dvtx
\mathbb{R}^{d} \rightarrow \mathbb{R}_+$ and an application $g\dvtx
\mathbb{R}_+ \rightarrow \mathbb{R}_+^*$. Our goal is to study the
ergodic behaviour of the self-interacting diffusion $X$ solution to
%
\begin{equation} \label{equX1}
    \mathrm{d}X_t = \mathrm{d}B_t - g(t)\nabla V(X_t - \overline{\mu}_t)\,\mathrm{d}t,\qquad
    X_0 = x,
\end{equation}
where $B$ is a standard Brownian motion and $\overline{\mu}_t$ denotes
the empirical mean of $X$:
%
\begin{eqnarray}\label{defmu}
\overline{\mu}_t = \frac{1}{r+t}\biggl(r\bar{\mu} + \int_{0}^{t}
X_{s}\,
\mathrm{d}s\biggr),\qquad\overline{\mu}_0 = \overline{\mu}.
\end{eqnarray}
Here $\mu $ is an initial (given) probability measure on
$\mathbb{R}^d$, $\bar{\mu}$ denotes the mean of $\mu$ and $r>0$ is an
initial weight (it permits us to consider any initial probability
measure).

First, note that for a quadratic interaction potential $V$, the process
satisfying \eqref{equX1} is exactly of the form of \eqref{eq:raimond}
and, in both cases, the occupation measure is penalized by $g(t)$.
Afterwards, a natural generalization of this process is the class of
self-interacting diffusions discussed here. The interesting point is
that we manage to study precisely the asymptotic behaviour of $X$ and
prove a convergence criterion. Moreover, this model could be used to
represent the behaviour of social insects; for instance, ant trails.
Indeed, ants mark their paths with pheromones. Certain ants lay down an
initial trail of pheromones as they return to the nest with food. This
trail attracts other ants and serves as a guide. As long as the food
source remains, the pheromone trail will be continually renewed.
Despite the quick evaporation, the path is reinforced and so, the ants
manage to gradually find the best route. In this (simplified) model,
the function $g$ reflects the speed of evaporation and $X$ denotes the
trail.

In order to study the solution to \eqref{equX1}, it is natural to
introduce the process $Y$, defined by
%
\begin{equation}\label{defY}
Y_t = X_t -\overline{\mu}_t.
\end{equation}
It is easily seen that $(Y_{t},t\geq 0)$ is the solution to the SDE
\begin{equation}\label{equYVV}
    \mathrm{d}Y_{t} = \mathrm{d}B_{t} - g(t) \nabla
V(Y_{t})\,\mathrm{d}t - Y_t\frac{\mathrm{d}t}{r+t},\qquad
    Y_{0} = x - \overline{\mu}
\end{equation}
and $\mathrm{d}\overline{\mu}_t = Y_t \frac{\mathrm{d}t}{r+t}.$ As $Y$
is a (non-homogeneous) Markov process, it is easier to study $Y$ than
$X$. Indeed, we will prove that $Y$ converges a.s. and satisfies the
pointwise ergodic theorem. Due to that, the behaviour of $X$ could seem
a bit easy at first glance. But it really shows unexpected behaviours
and, in particular, it does not satisfy the pointwise ergodic theorem
in general (because $\overline{\mu}_t$ does not converge, except for
functions $g$ going fast to infinity). This explains the difficulty of
studying more general self-interacting diffusions in non-compact spaces
(see \cite{AK}), driven by the generic equation
$\mathrm{d}X_t = \mathrm{d}B_t - \int_{\mathbb{R}^d} \nabla V(X_t,x)\,
\mathrm{d}\mu_t(x)\,\mathrm{d}t$.

The remainder of the paper is organized as follows. First, we enumerate
the hypotheses and state the main results in Section
\ref{s:assumptions}. We motivate our study, in Section \ref{s:quadra},
by the basic case when $V$ is quadratic, for which we have an explicit
expression of $X$ (in terms of Brownian martingale). Section
\ref{s:extrema} deals with the description of the behaviour of $Y$ near
the local extrema of $V$. Finally, Section \ref{s:X} is devoted to the
proof of the main results.

\section{Technical assumptions and main results}\label{s:assumptions}

In the sequel, $(\cdot,\cdot)$ stands for the Euclidian scalar product.
We also denote by $\mathcal{P}(\mathbb{R}^d)$ the set of probability
measures on $\mathbb{R}^d$.

Consider the potential $V\dvtx \mathbb{R}^d\rightarrow \mathbb{R}_+$. Let
$\mathit{Max} =\{M_{1},\ldots, M_{p}\}$ be the (finite) set of saddle points and
local maxima of $V$ and denote by $\mathit{Min}=\{m_{1},\ldots, m_{n}\} $ the
(finite) set of the local minima of $V$. So $\mathit{Min}\,\cup\,\mathit{Max}$ is the set
of critical points of $V$. We assume that $V$ is either quadratic
(Section \ref{s:quadra}) or:
\begin{enumerate}[(2)]
    \item[(1)] (\textit{regularity and positivity}) $V\in \mathcal{C}^2(\mathbb{R}^d)$ and $V> 0$;
    \item[(2)] (\textit{convexity}) $V= W+\chi$, where $\chi$ is a compactly supported function such that
    $\nabla \chi$ is $\tilde{C}$-Lipschitz (with $\tilde{C}>0$) and there exists $c>0$ such that $\nabla^2 W \geq c \mathit{Id}$;
    \item[(3)] (\textit{growth}) there exists $a>0$ such that for all $x\in \mathbb{R}^d$, we have
    %
    \begin{equation}\label{growth}
    \Delta V(x)\leq aV(x)\quad\mbox{and}\quad\lim_{|x|\rightarrow \infty} \frac{|\nabla V(x)|^2}{V(x)} = \infty;
    \end{equation}
    \item[(4)] (\textit{critical points})   $\forall m_i$, $\forall \xi \in \mathbb{R}^d$, $(\nabla^2
V(m_i)\xi,\xi)>0$ and for all $M_i$, $\nabla^2 V(M_i)$ admits a~negative eigenvalue.
\end{enumerate}

\begin{rk}
By the growth condition \eqref{growth}, $|\nabla V|^2 - \Delta V$ is
bounded by below.
\end{rk}

Suppose also that $g\dvtx \mathbb{R}_+ \rightarrow \mathbb{R}_+$ is
non-decreasing and $g\in \mathcal{C}^1(\mathbb{R}_+)$. We denote by
$g(\infty)$ the limit of $g(t)$ and we exclude the trivial case where
$g$ is identically zero, so that $g(\infty)>0$. Let $G(t) :=\int_0^t
g(s)\,\mathrm{d}s$ and $G^{-1}$ be its generalized inverse: $G^{-1}(t)
:= \inf \{u\ge 0; G(u) \ge t\}$.

\begin{rk}\label{rk:G}
If $g(\infty) =\infty$, then for all $T>0$, we have that $G^{-1}(t+T) -
G^{-1}(t) \mathop{\longrightarrow}\limits_{t\rightarrow\infty} 0$.
\end{rk}

The following easy result will be very useful in the sequel.
\begin{lemma}\label{l:ipp}
Suppose that $g'(t)/g^2(t)$ converges to 0. Then the following hold:
\begin{longlist}
\item for any $c>0$, $\int_0^t s^2 \mathrm{e}^{2cG(s)}\,\mathrm{d}s = \mathrm{O}(t^2
\mathrm{e}^{2cG(t)}/g(t))$;

\item if $g(\infty)=\infty$, then we have $\int_0^t \frac{g'(s)}{g(s)^2}
G(s)\,\mathrm{d}s = \mathrm{O}(t)$;

\item for $H(t) := \int_{0}^{t} \frac{\mathrm{e}^{-cG(u)}}{(r+u)^{2}}\,
\mathrm{d}u$, the following expansion holds:
\[
H(t) = H(\infty) - \frac{1}{c g(t)(r+t)^{2}}\mathrm{e}^{-cG(t)} +
\mathrm{o}\biggl(\frac{\mathrm{e}^{-cG(t)}}{t^{2}g(t)}\biggr).
\]
\end{longlist}
\end{lemma}
\begin{pf}
We deduce all these estimates from an integration by parts:
\[
\int_0^t
s^2 \mathrm{e}^{2cG(s)} \,\mathrm{d}s = \frac{t^2 \mathrm{e}^{2cG(t)}}{2cg(t)} - \int_0^t
\biggl(\frac{s}{g(s)} - \frac{s^2 g'(s)}{2g(s)^2}\biggr)
\frac{\mathrm{e}^{2cG(s)}}{c}\,
\mathrm{d}s =\mathrm{O}\bigl(t^2 \mathrm{e}^{2cG(t)}/g(t)\bigr),
\]
and we obtain $H(t) - H(s) =
\frac{\mathrm{e}^{-cG(s)}}{r+s} - \frac{\mathrm{e}^{-cG(t)}}{r+t} - c\int_{s}^{t}
g(u)\mathrm{e}^{-cG(u)}\frac{\mathrm{d}u}{r+u}$. Similarly, for $t$ large enough
and $u$ such that $g(u)>0$, we find $\int_u^t \frac{g'(s)}{g(s)^2}
G(s)\,
\mathrm{d}s = -\frac{G(t)}{g(t)} + \frac{G(u)}{g(u)} + t-u =\mathrm{O}(t)$.
\end{pf}

\subsection{Existence}
We begin by proving that the SDE admits a unique global strong
solution.
\begin{propo}
For any $x\in \mathbb{R}^d$, $\mu\in\mathcal{P}(\mathbb{R}^d)$ and
$r>0$, there exists a unique global strong solution $(X_t, t\geq 0)$ of
\eqref{equX1}.
\end{propo}
\begin{pf}
The local existence and uniqueness of the solution to \eqref{equX1} is
standard. We only need to prove here that $Y$, hence $X$ (since $X_t :=
Y_t + \int_0^t Y_s \frac{\mathrm{d}s}{r+s}$), does not explode in a
finite time. To this aim, apply It\^{o}'s formula to the function
$x\mapsto V(x)$:
\[
\mathrm{d}V(Y_t) = (\nabla V(Y_t),\mathrm{d}B_t) +\biggl(\frac{1}{2}\Delta
V(Y_t) - g(t) |\nabla V(Y_t)|^2 - \frac{1}{r+t}(\nabla V(Y_t),Y_t)
\biggr)\,
\mathrm{d}t,
\]
and introduce the sequence of stopping times $\tau_0=0$ and
\[
\tau_n = \inf\biggl\{t\geq 0; V(Y_t) +\int_0^t g(s)|\nabla V(Y_s)|^2\,
\mathrm{d}s >n\biggr\}.
\]
By the convexity condition, we have $(\nabla
V(y),y) \mathop{\longrightarrow}\limits_{|y|\rightarrow +\infty} +\infty$ and
so the growth condition \eqref{growth} implies the existence of $C$
such that $\mathbb{E}V(Y_{t\wedge \tau_n}) \leq \mathbb{E}V(Y_0) + \mathrm{e}^{C
t}.$
\end{pf}

\subsection{Results}
We give now a description of the asymptotic behaviour of both $\mu_t$
and $X_t$.
\begin{defn} The process $X$ satisfies the pointwise ergodic
theorem if there exists a measure $\mu_\infty$ such that a.s. $\mu_t :=
\frac{1}{r+t}(r\mu+\int_0^t \delta_{X_s}\,\mathrm{d}s) \rightarrow
\mu_\infty$ for the weak convergence of measures: For all continuous
bounded functions $f$, $\frac{1}{t}\int_0^t f(X_s)\,\mathrm{d}s
\stackrel{\mathrm{a.s.}}{\longrightarrow}  \int f\,\mathrm{d}\mu_\infty$.
\end{defn}

\begin{theo}
\begin{enumerate}[(1)]
    \item[(1)] The process $Y$ satisfies the pointwise ergodic theorem:
    Almost surely, the empirical measure of $Y$ converges weakly to a measure, which is a convex combination of Dirac
    measures taken in the local minima of $V$.
    \item[(2)] The process $X$ satisfies the pointwise ergodic theorem
    if and only if the mean process $\overline{\mu}_t$ converges almost surely.
\end{enumerate}
\end{theo}

A necessary condition for the convergence of $\overline{\mu}_t$ is that
0 is the unique minimum of $V$. We will prove this result in Section
\ref{ss:ergoY}. Indeed, what we need here is not only the convergence
of $Y_t$ to zero, but the convergence of the integral $\int_0^t Y_s
\frac{\mathrm{d}s}{r+s}$, which depends on the speed of convergence of
$Y_t$. The main result of this paper is the following description of
the asymptotic behaviour of~$X$, shown in Section \ref{ss:psX}:
\begin{theo}
Suppose that $\sqrt{g(t)^{-1}\log G(t)} = \mathrm{O}(h(t)^{-1})$, where $G$ is a
primitive of $g$ and $\int_0^\infty \frac{\mathrm{d}s}{(1+s) h(s)}
<\infty$.
\begin{enumerate}[(1)]
    \item[(1)] The process $Y$ converges almost surely to $Y_\infty$,
    where $Y_\infty$ belongs to the set of the local minima of $V$.
    For each local minimum $m$ of $V$, one has
    $\mathbb{P}(Y_\infty = m)>0$.
    \item[(2)] On the set $\{Y_\infty =0\}$, both $X_t$ and
    $\overline{\mu}_t$ converge almost surely to $\overline{\mu}_\infty := \overline{\mu}+ \int_0^\infty Y_s
    \frac{\mathrm{d}s}{r+s}$, whereas on the set $\{Y_\infty \neq
    0\}$, we have that $\lim_{t\rightarrow\infty} \frac{X_t}{\log t} =
    Y_\infty$.
\end{enumerate}
\end{theo}

\section{A motivating example}\label{s:quadra}
We consider $V(x) = \frac{1}{2}(x, c x)$, where $c$ is a symmetric
positive definite matrix. Let $X$ be the solution of the SDE
\begin{equation} \label{equX2}
    \mathrm{d}X_t = \mathrm{d}B_t - g(t)\biggl(cX_t  - \frac{r}{r+t} c\overline{\mu} - \frac{1}{r+t}\int_0^t c
    X_s\,
\mathrm{d}s \biggr) \,\mathrm{d}t,\qquad      X_0 = x.
\end{equation}
Without any loss of generality, we suppose that $d =1$. When $d\ge 1$,
it suffices to diagonalize the matrix $c$ and to remark that, for an
orthogonal matrix $U$, the process $(U\cdot B_s, s\geq 0)$ is also a~Brownian motion.

\subsection{Explicit expression of $X$}

\begin{lemma}
If $X$ is the solution to \eqref{equX2}, then we have
\[
Y_{t} := X_t-
\bar{\mu}_t = \frac{\mathrm{e}^{-cG(t)} }{r+t}
\biggl(\int_{0}^{t}(r+s)\mathrm{e}^{cG(s)}\,\mathrm{d}B_{s}+r(x-\overline{\mu}) \biggr).
\]
\end{lemma}
\begin{pf}
The process $Y$ satisfies
\begin{equation}\label{equY}
\mathrm{d}Y_t = \mathrm{d}B_t - \biggl( cg(t)+ \frac{1}{r+t}\biggr)
Y_t\,
\mathrm{d}t,\qquad Y_0 = x - \overline{\mu}.
\end{equation}
To express $Y$ in terms of a Brownian martingale, we consider the
function of $Y$ defined by $U_{t} := (r+t)\mathrm{e}^{cG(t)} Y_{t}$. Then
It\^{o}'s formula implies
\[
 \mathrm{d}U_t = (r+t) \mathrm{e}^{cG(t)}\,\mathrm{d}B_{t},\qquad  U_0 = r(x - \overline{\mu}).
\]
\upqed\end{pf}

\begin{coro} \label{formuleX}
Let $F(t) = \int_0^t \mathrm{e}^{-cG(s)} \frac{g(s)}{r+s}\,\mathrm{d}s$. The
solution to the SDE \eqref{equX2} is given by
\[
X_t = x +
rc(\overline{\mu} -x) F(t) + \int_0^t \bigl[1 - (r+s) c\mathrm{e}^{cG(s)}\bigl(F(t) -
F(s)\bigr) \bigr]\,\mathrm{d}B_s.
\]
\end{coro}
\begin{pf}
Remark that $\frac{\mathrm{d}}{\mathrm{d}t}\overline{\mu}_{t} =
\frac{Y_{t}}{r+t} $. So, by Fubini's theorem for stochastic integrals
(see \cite{ReY}, page~175), we have
\[
\overline{\mu}_{t}= \int_{0}^{t}
(r+s) \mathrm{e}^{cG(s)} \bigl(H(t) -H(s)\bigr)\, \mathrm{d}B_{s} + r(x-\bar{\mu})H(t) +
\bar{\mu}
\]
with $ H(t) = \int_{0}^{t} \frac{\mathrm{e}^{-cG(u)}}{(r+u)^{2}}\,
\mathrm{d}u = \frac{1}{r} - cF(t) - \frac{\mathrm{e}^{-cG(t)}}{r+t}$. As $X_t =
Y_t +\overline{\mu}_t$, the latter result implies the desired
expression.
\end{pf}

\subsection{Ergodic result}
We begin to prove the pointwise ergodic theorem for the following
non-homogeneous\break \mbox{(Gauss-)Markov} process.

\begin{propo}\label{ergodic} Let $a \dvtx \mathbb{R}_+ \rightarrow
\mathbb{R_+}$ be a continuous function, $A(t) := \int_0^t a(s)\,
\mathrm{d}s$ and $K(t) := \mathrm{e}^{-2A(t)}\int_0^t \mathrm{e}^{2A(s)}\, \mathrm{d}s$.
Suppose that $a(\infty) = \lim_{t\rightarrow\infty} a(t)$ exists and is
non-zero, so that $K(\infty)= \frac{1}{2a(\infty)} < \infty$. Consider
the process
\[
\mathrm{d}Z_t = -a(t) Z_t\,
\mathrm{d}t + \mathrm{d}B_t,\qquad Z_0 = z.
\]
Then, denoting by $\gamma$ the
centered Gaussian measure with variance $K(\infty)$ (with the
convention $\gamma=\delta_0$ for $K(\infty)=0$), we have for all
continuous bounded functions $\varphi$
\[
\frac{1}{t}\int_0^t \varphi(Z_s)\,\mathrm{d}s \mathop{\stackrel{a.s.}{\longrightarrow} }_{t\rightarrow\infty}
\int \varphi(z) \gamma(\mathrm{d}z).
\]
\end{propo}
\begin{pf}
We prove the result for the Fourier transform. First, note that
\[
Z_{t}= \mathrm{e}^{-A(t)} \biggl(\int_{0}^{t} \mathrm{e}^{A(s)}\,\mathrm{d}B_{s} + z\biggr).
\]
Let $\mathcal{F}_s := \sigma (B_u, 0\leq u\leq s)$. Knowing $\mathcal{F}_s$,
$Z_t$ has the Gaussian distribution with mean $m(s,t) := \mathrm{e}^{-(A(t) - A(s))}Z_s$ and
variance $K(s,t) := \mathrm{e}^{-2A(t)}\int_s^t \mathrm{e}^{2A(u)}\,\mathrm{d}u$. Fix
$t\in \mathbb{R}_+,u\in \mathbb{R}$ and define the martingale
$M_s^{t,u} := \mathbb{E}( \mathrm{e}^{\mathrm{i}uZ_t}| \mathcal{F}_s ) = \exp\{\mathrm{i}um(s,t) -
\frac{u^2}{2}K(s,t)\}$. Applying It\^{o}'s formula to $s\mapsto
M_s^{t,u}$, we find that $\mathrm{d}M_s^{t,u} = \mathrm{i}u \mathrm{e}^{-(A(t)-A(s))}
M_s^{t,u}\,\mathrm{d}B_s$, and so
\[
\mathrm{e}^{\mathrm{i}uZ_t} = \mathbb{E}\mathrm{e}^{\mathrm{i}uZ_t} +
\int_0^t \mathrm{i}u\mathrm{e}^{-(A(t) - A(s))}M_s^{t,u}\,\mathrm{d}B_s.
\]
Then, by
Fubini's theorem for stochastic integrals, we easily obtain
\begin{equation}\label{quadraFub}
\int_0^t \mathrm{e}^{\mathrm{i}uZ_s}\,\mathrm{d}s = \int_0^t
\mathbb{E}\mathrm{e}^{\mathrm{i}uZ_s}\,
\mathrm{d}s + \int_0^t \mathrm{d}B_s \int_s^t
\mathrm{i}u\mathrm{e}^{-(A(r)-A(s))}M_s^{r,u}\,\mathrm{d}r.
\end{equation}
As $Z_t$ is Gaussian with variance $K(0,t)$, it converges in
distribution to a Gaussian variable of law $\gamma =
\mathcal{N}(0,K(\infty))$. Because of Ces\`{a}ro's result, we have
\[
\lim_{t\rightarrow \infty} \frac{1}{t} \int_0^t \mathbb{E}
\mathrm{e}^{\mathrm{i}uZ_s}\,
\mathrm{d}s = \mathrm{e}^{-(u^2/2)K(\infty)}.
\]
We wish to find an
asymptotic equivalent to the stochastic factor of \eqref{quadraFub}. To
this aim, consider $N_{s,t}^u(v) := \int_s^t \mathrm{i}u \mathrm{e}^{A(v)-A(r)}
M_v^{r,u}\,\mathrm{d}r$. First, on the set $\{\int_0^\infty
\langle N^u_{\cdot,t}(s)\rangle_s\,\mathrm{d}s <\infty\}$, the local martingale
$\int_0^t N_{s,t}^u(s)\, \mathrm{d}B_s$ converges a.s. to a finite
variable and thus is of the order of $\mathrm{o}(t)$. Actually, we decompose it
as
\begin{equation}\label{quadraFub2}
\int_0^t N_{s,\infty}^u(s)\,\mathrm{d}B_s - \int_0^t
N^u_{t,\infty}(s)\,\mathrm{d}B_s.
\end{equation}
On the set $\{\int_0^\infty \langle N^u_{\cdot,t}(s)\rangle _s \,\mathrm{d}s
=\infty\}$, the LLN for martingales implies a.s.
\[
\int_0^t
\mathrm{d}B_s \int_s^{\infty} \mathrm{i}u\mathrm{e}^{-(A(r)-A(s))}M_s^{r,u}\,\mathrm{d}r =
\mathrm{o}\biggl(\int_0^t \biggl|\int_s^{\infty} \mathrm{i}u\mathrm{e}^{-(A(r)-A(s))} M_s^{r,u}\,\mathrm{d}r
\biggr|^2\,\mathrm{d}s\biggr).
\]
Indeed, we obtain the following upper bound by
using the initial definition of $M_s^{r,u}$:
\[
|N_{s,t}^u(s)|\leq |u|\int_s^t \mathrm{e}^{A(s)-A(r)}\,\mathrm{d}r = |u|\mathrm{e}^{A(s)}
(I_t -I_s),
\]
where $I_t := \int_0^t \mathrm{e}^{-A(r)}\, \mathrm{d}r = I_\infty
-\frac{\mathrm{e}^{-A(t)}}{a(t)}+ \mathrm{o}(\frac{\mathrm{e}^{-A(t)}}{a(t)} )$, we find by the
triangle inequality that $\int_0^t \mathrm{e}^{2A(s)}(I_t-I_s)^2\,\mathrm{d}s
=\mathrm{O}(t)$. So we have
\[
\mathbb{E}\biggl(\int_0^t N_{t,\infty}^u(s) \,\mathrm{d}B_s\biggr)^2 = \int_0^t
\mathbb{E} (N_{t,\infty}^u(s)^2)\,\mathrm{d}s \leq |u|^2\int_0^t
\mathrm{e}^{2A(s)}\,\mathrm{d}s (I_\infty-I_t)^2 = \mathrm{O}(1).
\]
Borel--Cantelli's lemma permits us to conclude that
$\frac{1}{t}\int_0^t N_{t,\infty}^u(s)\,\mathrm{d}B_s$ converges a.s. to
0.
\end{pf}

\begin{theo} \label{thergoY}
Suppose that $g'(t)/g^2(t)$ converges to 0. Then,
with probability 1, the empirical measure $\mu_t=\frac{r}{r+t}\mu +
\frac{1}{r+t}\int_0^t \delta_{X_s}\,\mathrm{d}s$ converges weakly to
$\mu_\infty$. Moreover, the mean $\overline{\mu}_t = \frac{1}{r+t}
\int_0^t X_s\,\mathrm{d}s + \frac{r}{r+t} \overline{\mu}$ also converges
almost surely.
\end{theo}

\begin{pf}
We remind the reader that $g'(t)/g^2(t)$ converges to 0. We start by
proving that $\bar{\mu}_t$ converges a.s. Decompose the process
$\overline{\mu}_t = \overline{\mu}_t^1 + \overline{\mu}_t^2 +
\overline{\mu}_t^3$, where
\begin{eqnarray*}
\overline{\mu}_t^1 &=& \overline{\mu} + r(x-\overline{\mu}) H(t),\\
\overline{\mu}_t^2 &=& \bigl(H(t) - H(\infty)\bigr) \int_{0}^{t} (r+s)
\mathrm{e}^{cG(s)}\,\mathrm{d}B_{s},\\
\overline{\mu}_t^3 &=& \int_{0}^{t} (r+s)\mathrm{e}^{cG(s)} \bigl(H(\infty) - H(s)
\bigr)\,
\mathrm{d}B_{s}.
\end{eqnarray*}
The convergence of $H(t)$ obviously implies the convergence of
$\overline{\mu}_t^1$. The deterministic factor of~$\overline{\mu}_t^2$
is equivalent to $\frac{1}{cg(t)t^2} \mathrm{e}^{-cG(t)}$ and, due to Lemma
\ref{l:ipp}, the quadratic variation of the stochastic factor in
$\overline{\mu}_t^2$ is of the order of $\frac{t^2 \mathrm{e}^{2cG(t)}}{g(t)}$.
By Lemma \ref{l:ipp} and the law of the iterated logarithm (\cite{Lep},
Theorem 3), we have $\overline{\mu}_t^2
\mathop{\stackrel{\mathrm{a.s.}}{\longrightarrow}}\limits_{t\rightarrow\infty} 0$. Finally,
$\overline{\mu}_t^3$ is a $L^2$-bounded martingale and thus converges
a.s. Putting all the pieces together, we conclude that
$\overline{\mu}_t \mathop{\stackrel{\mathrm{a.s.}}{\longrightarrow} }\limits_{t\rightarrow\infty}
\overline{\mu} + H(\infty) r(x-\overline{\mu}) + \int_{0}^{\infty}
(r+s)\mathrm{e}^{cG(s)} (H(\infty) - H(s) )\,\mathrm{d}B_{s}$.

To show that $\mu_t$ converges, we first point out that the
deterministic factor of $X_t$ converges. Decompose the process $X$ into
three parts: $X_t = \overline{\mu}_\infty + \phi(t) U_t + \mathrm{o}(1),$ where
\begin{eqnarray*}
\overline{\mu}_\infty &:=& x + c r (\overline{\mu} -x) F(\infty)+
\int_0^\infty \bigl[1 - (r+s)c\mathrm{e}^{cG(s)}\bigl(F(\infty) -
F(s)\bigr) \bigr]\,\mathrm{d}B_s,\\
U_t &:=& \frac{\mathrm{e}^{-cG(t)}}{r+t}\int_0^t (r+s)\mathrm{e}^{cG(s)}\,\mathrm{d}B_s,\\
\phi(t) &:=& c(r+t) \bigl(F(\infty) - F(t)\bigr)\mathrm{e}^{cG(t)}.
\end{eqnarray*}
Again, we prove the result for the Fourier transform. We have the
following:
\[
\frac{1}{t} \int_0^t \mathrm{e}^{\mathrm{i}uX_s}\,\mathrm{d}s =
\frac{\mathrm{e}^{\mathrm{i}u(\overline{\mu}_\infty+\mathrm{o}(1))}}{t}\int_0^t
\mathrm{e}^{\mathrm{i}u\phi(s)U_s}\,
\mathrm{d}s.
\]
By Proposition \ref{ergodic}, $\phi(t)U_t$ is ergodic. So
$\frac{1}{t}\int_0^t \mathrm{e}^{\mathrm{i}u\phi(s)U_s}\,\mathrm{d}s$ converges
a.s.
\end{pf}

\begin{coro}
Suppose that $g'(t)/g^2(t)$ converges to 0 and $g(\infty) <\infty$.
Then the limit $\lim \mu_t = \mu_\infty$ is the Gaussian measure
$\mu_\infty = \mathcal{N}(\overline{\mu}_\infty,
\frac{1}{2g(\infty)c})$.
\end{coro}

\subsection{Asymptotic behaviour of $X$}
We prove here that, depending on $g$, $X$ exhibits three different
behaviours: $X$ converges either almost surely, or in probability (and
not a.s.), or it diverges.
 First, we describe roughly the asymptotic behaviour of $X$.
\begin{propo} Suppose that
$g(\infty) <\infty$. Then we have
\[
\mathbb{P}\Bigl(\limsup_{t\rightarrow\infty} X_t =
+\infty\Bigr) = \mathbb{P}\Bigl(\liminf_{t\rightarrow \infty} X_t =
-\infty\Bigr)=1.
\]
\end{propo}
\begin{pf}
Let $A$ be a non-negligible subset of $\mathbb{R}$. We have the
asymptotic equivalence
\[
\int_0^t \delta_{X_s}(A)\,\mathrm{d}s
\mathop{\sim}_{t\rightarrow \infty} t l,
\]
where $l$ is a positive
constant depending on $A$. So, $\int_0^\infty \delta_{X_s}(A)\,
\mathrm{d}s = \infty$ a.s. and $\mu_\infty$ is diffusive. It then
implies that for all $K>0$, $\int_0^\infty \delta_{X_s}([K,\infty[)\,
\mathrm{d}s = \infty$ a.s. and so
\[
\mathbb{P}\biggl(\bigcap_{K\geq 1} \biggl\{\int_0^\infty \1_{\{X_s\geq K\}}\,
 \mathrm{d}s = \infty \biggr\}\biggr) = 1.
\]
 We conclude that $\mathbb{P}(\limsup_{t\rightarrow\infty}
X_t = + \infty) =1$. The proof is the same for $\liminf_{t\rightarrow
\infty}X_t$.
\end{pf}

\begin{propo}
Suppose that $g'(t)/g^2(t)$ converges to 0 and $g(\infty) =\infty$.
Then $X_{t}$ converges in probability to a random variable $X_\infty$
and a.s. $\mu_t$ converges weakly to $\delta_{X_\infty}$.
\end{propo}
\begin{pf}
As $Y$ is Gaussian and $\mathbb{E}(Y_t^2)= \mathrm{O} (g(t)^{-1})$, it converges
in $L^2$ and so in probability to 0. Decomposing $X$ as $X_t = Y_t +
\int_0^t Y_s \frac{\mathrm{d}s}{r+s}$, it remains to show that the
previous integral converges in probability. Using the explicit form of
$Y$, Fubini's theorem for stochastic integrals ensures
\[
\int_0^t Y_s
\frac{\mathrm{d}s}{r+s} = r(x-\bar{\mu})\int_0^t
\mathrm{e}^{-cG(s)}\frac{\mathrm{d}s}{(r+s)^2} + \int_0^t (r+u)\mathrm{e}^{cG(u)}
\biggl(\int_u^t \mathrm{e}^{-cG(s)}\frac{\mathrm{d}s}{(r+s)^2}\biggr)\,\mathrm{d}B_u.
\]
The
quadratic variation of the Brownian integral converges by Lemma
\ref{l:ipp} and thus $X$ converges to $X_\infty$ in $L^2$. Remark that
the law of the iterated logarithm does not imply here that $X$
converges a.s. since we do not know whether $\log G(t) /g(t)$ converges
to 0 or not. We then easily have that $\mu_t$ converges toward
$\delta_{X_\infty}$ in probability. By Theorem \ref{thergoY}, a.s.
$\mu_t$ converges (weakly) and so we conclude.
\end{pf}

\begin{propo}
Suppose that $g'(t)/g^2(t)$ converges to 0 and $g(t)^{-1} \log G(t)$ is
bounded for $t$ large enough. Then there exists $C>0$ such that
\[
\mathbb{P}\Bigl(\limsup_{t\rightarrow\infty}\vert Y_t\vert \leq C\Bigr) =1.
\]
\end{propo}
\begin{pf}
We write $Y$ as a Brownian (local) martingale: $Y_t =
\frac{\mathrm{e}^{-cG(t)}}{r+t}(Y_0 + \int_0^{t}
(r+s)\mathrm{e}^{cG(s)}\,
\mathrm{d}B_{s})$. To estimate the quadratic variation of $Y$, we use
Lemma \ref{l:ipp} and thus, by the law of the iterated logarithm, there
exists $C$ such that a.s. $\limsup_{t\rightarrow \infty} \vert Y_t\vert
\leq C$.
\end{pf}

\begin{coro}
Suppose that $g(t)^{-1} \log G(t)$ is lower bounded away from 0 and
upper bounded for $t$ large enough. Then $X_{t}$ is bounded a.s.,
converges in probability (but not a.s.) to $X_\infty =
\overline{\mu}_\infty$ and a.s. $\mu_t$ converges weakly to
$\delta_{X_\infty}$.
\end{coro}
\begin{pf}
$Y$ is a.s. bounded and $\overline{\mu}_t$ converges a.s., so $X_t =
Y_t + \overline{\mu}_t$ is also a.s. bounded. As $Y_t$ is Gaussian, it
converges (in law) to a centered Gaussian variable. The latter being
bounded, $Y_t$ converges in probability to 0. By the law of the
iterated logarithm, $Y_t$ does not converge a.s. to~0 (since $\log G(t)
/g(t)>0$ for large $t$). So, $X_t$ converges in probability to
$\overline{\mu}_\infty$. We conclude by uniqueness of the limit that
a.s. $\mu_t$ converges weakly to $\delta_{\overline{\mu}_\infty}$.
\end{pf}

\begin{propo}
Suppose that $g'(t)/g^2(t)$ converges to 0 and
$\lim_{t\rightarrow\infty} g(t)^{-1} \log G(t) =0$. Then the
process $Y_{t} := X_{t} - \overline{\mu}_{t}$ converges to $0$ a.s.
Moreover, both $ X_{t}$ and $\overline{\mu}_t$ converge to
$\overline{\mu}_{\infty}$ a.s. and a.s. $\mu_t$ converges weakly to
$\delta_{\overline{\mu}_{\infty}}$.
\end{propo}
\begin{pf}
We only need to prove that $Y_t := X_t - \overline{\mu}_t$ converges
a.s. to 0. We have already seen that $Y_{t} = \frac{\mathrm{e}^{-cG(t)}}{r+t}
\int_{0}^{t} (r+s)\mathrm{e}^{cG(s)}\,\mathrm{d}B_{s} + r(x-\overline{\mu})
\frac{\mathrm{e}^{-cG(t)}}{r+t} =: U_{t} + v_{t}$. The deterministic term $v_t$
converges obviously to $0$ and the law of the iterated logarithm
implies that $U_{t}$ converges a.s. to~$0$. By uniqueness of the limit
of $\mu_t$, we conclude that $\mu_\infty =
\delta_{\overline{\mu}_\infty}$.
\end{pf}

\section{Study of $Y$ with respect to the critical points of $V$} \label{s:extrema}
From now on, we assume that $g'(t)/g^2(t)$ converges to 0 (this
hypothesis is only needed to study the behaviour of $Y$ near a local
minimum of $V$). We study the process $Y_t = X_t - \bar{\mu}_t$, which
is the solution to
\begin{equation}\label{equYV}
\mathrm{d}Y_t = \mathrm{d}B_t - \biggl(g(t)\nabla
V(Y_t)+\frac{Y_t}{r+t}\biggr)\,
\mathrm{d}t;\qquad Y_0=x-\overline{\mu}.
\end{equation}
More precisely, we study the behaviour of $Y$ near the critical points
of $V$. We show in particular, for each local minimum of $V$, that $Y$
stays close to it with positive probability, whereas this probability
is zero for any unstable critical point.

\subsection{Behaviour near the critical points of $V$}
\begin{propo} \label{propVcritique}
Almost surely, $\forall \varepsilon >0$, $\forall t>0$,
\[
T_{t}^{\varepsilon} := \inf \{ s \geq t ; d(Y_{s}, \mathit{Min} \cup \mathit{Max}) <
\varepsilon \} <\infty.
\]
\end{propo}
\begin{pf}
Let $\varepsilon > 0$. Applying It\^{o}'s formula to $x\mapsto V(x)$,
we obtain
\[
\mathrm{d}V(Y_{t})= (\nabla V(Y_{t}),\mathrm{d}B_{t}) - \biggl(g(t)\vert
\nabla V(Y_t) \vert^{2} + \frac{1}{r+t} (Y_t,\nabla V(Y_t))-
\frac{1}{2}\Delta V(Y_t) \biggr)\,\mathrm{d}t. \label{equdVYD}
\]
It then follows from the growth condition \eqref{growth} that on the
set $\{z ; d(z,\mathit{Min} \cup \mathit{Max}) > \varepsilon \}$ and for $t \geq 0$, the
function  $y \mapsto \frac{1}{r+t}(y,\nabla V(y)) + g(t)\vert \nabla
V(y)\vert^{2} - \frac{1}{2}\Delta V(y)$ is bounded from below. So there
exists $C>0$ such that, for $\forall y \in \{z;d(z,\mathit{Min} \cup \mathit{Max}) >
\varepsilon \}$, we have
\begin{equation}\label{minD2}
g(t) \vert \nabla V(y)\vert^{2} + \frac{1}{r+t}(y,\nabla V(y))-
\frac{1}{2}\Delta V(y) \geq \biggl(g(t)-\frac{g(\infty)}{2}\biggr) \vert \nabla
V(y)\vert^{2} -C.
\end{equation}
Let us introduce the stopping time $T_{t}^{\varepsilon} = \inf\{s \geq
t ; d(Y_{s}, \mathit{Min} \cup \mathit{Max}) < \varepsilon \}$ and prove that
$\mathbb{P}(T_{t}^{\varepsilon} < + \infty)=1.$ It follows from
\eqref{minD2} that there exists $t_0$ such that, for $t>t_0$,
$(V(Y_{s\wedge {T_{t}^{\varepsilon}}}) + C (s\wedge
{T_{t}^{\varepsilon}}), s \geq t)$ and
\[
\biggl(V(Y_{s\wedge
{T_{t}^{\varepsilon}}} ) + C (s\wedge {T_{t}^{\varepsilon}}) +
\int_{0}^{s \wedge T_{t}^{\varepsilon}} \biggl( g(u) - \frac{1}{2} g(\infty)
\biggr) \vert\nabla V(Y_{u\wedge T_{t}^{\varepsilon}})\vert^{2}\,\mathrm{d}u,
s \geq t \biggr)
\]
are two super-martingales. As they are positive, they
converge a.s. (as $s\rightarrow \infty$). So, the process $(\int_{0}^{s
\wedge T_{t}^\varepsilon} g(u) \vert \nabla V(Y_{u\wedge
T_{t}^{\varepsilon}})\vert^{2}\,\mathrm{d}u, s \geq t )$ also converges
a.s. On the set $\{ T_{t}^{\varepsilon} = +\infty \}$, we have
\[
\vert \nabla V(Y_{s\wedge T_{t}^{\varepsilon}}) \vert^2
\mathop{\stackrel{\mathrm{a.s.}}{\longrightarrow}}_{s\rightarrow\infty}0.
\]
Thus
$Y_{s\wedge T_{t}^{\varepsilon}}$ gets close to $\mathit{Min} \cup \mathit{Max}$ and
there is a contradiction. Finally, $\mathbb{P}(T_{t}^{\varepsilon} < +
\infty )=1$ for all $t> t_0$. For $t\leq t_0$, we conclude since $t
\mapsto T_t^\varepsilon$ is increasing.
\end{pf}

\begin{coro} \label{propVcritique2}
Almost surely, the sequence of stopping times $T_n := \inf\{s >n ;
d(Y_{s},\mathit{Min}\cup \mathit{Max})< \varepsilon\}$ satisfies: $T_n\rightarrow
\infty$, and $\forall n \geq 1$, $\mathbb{P}(T_{n} <+\infty)=1$ and
$d(Y_{T_{n}},\mathit{Min}\cup \mathit{Max}) < \varepsilon .$
\end{coro}

\subsection{Case of a stable critical point: Local minimum}
We will prove that if $Y_0$ is near a local minimum $m$, then the set
$\{Y_{s} ; s \geq 0\}$ stays near $m$ with a positive probability.
Indeed, a second-order Taylor expansion permits us to compare
$(y-m,\nabla V(y))$ with $\vert y-m\vert^{2}$ and we use a comparison
theorem. Let $m$ be a local minimum of $V$ such that $\nabla^2 V(m)
>0$. By Taylor's formula, there exist $a>0$ and $\varepsilon_{0} > 0$
such that for all $|y-m| \leq \varepsilon_{0}$ we have $(y-m,\nabla
V(y)) \geq a|y-m|^{2}.$ Without any loss of generality, we suppose
$m=0$ in the proofs.

\begin{propo} \label{propcvYai}
Suppose that $g(t)^{-1}\log G(t)$ is bounded on $\mathbb{R}_+$. Let
$\varepsilon_{0}> \varepsilon> 0$. Then, there exists a positive
stopping time $T_{0}$ such that for all $T > T_0$, we have on the event
$\{|Y_{T}-m| < \varepsilon\}$ that $\mathbb{P}(\forall s \geq 0; \vert
Y_{s+T}-m \vert < \varepsilon)>0$. Moreover, for any $T>T_0$, we have
on the event $\{\forall s \geq T; \vert Y_{s}-m \vert < \varepsilon
\}$:
\[
\vert Y_{t+T}- m \vert = \mathrm{O}\bigl(\sqrt{g(t+T)^{-1} \log G(t+T)}\bigr)\qquad
a.s.
\]
\end{propo}
\begin{pf}
Consider the time-shifted process $\widetilde{Y_{t}} := Y_{t+T}$. Let
$\varepsilon>0$. We will construct a one-dimensional process $U$ such
that for all $t\ge 0$, we have a.s. $|\widetilde{Y}_t| \le U_t$.

(1) Suppose that $d=1$. As $V{''}(m)> 0$, there exists $a>0$ such that
for all $|y| \le \varepsilon$, $yV{'}(y) \geq ay^2$. Introduce the
non-negative process $U$, a unique solution to the SDE
\begin{equation} \label{equU}
\mathrm{d}U_{t} = \sign\!{(\widetilde{Y}_t)}\,\mathrm{d}B_{t}^{T} -
ag(t+T)U_{t}\,\mathrm{d}t + \mathrm{d}L_{t},\qquad U_0=|\widetilde{Y}_0|,
\end{equation}
where $L$ is the local time of $U$ in 0. Let $\alpha(t)$ be the
function such that $\int_{0}^{\alpha(t)} \mathrm{e}^{2aG(s+T)}\,\mathrm{d}s = t $
and $\alpha(0)=0$. Then, the process $A_{t}:= \int_{0}^{\alpha(t)}
\mathrm{e}^{aG(s+T)}\,\mathrm{d}L_{s}$ is the local time in zero of
$W_t=\int_0^{\alpha(t)} \mathrm{e}^{aG(s+T)}\sign{(\widetilde{Y}_s)}\,
\mathrm{d}B_s^T$. Denote by $W^{+}$ the reflected Brownian motion
associated to $W$. Skorokhod's lemma (see \cite{HaSh}) then entails
that $\mathrm{e}^{aG(\alpha(t)+T)} U_{\alpha(t)} = W_t^+$. So the (strong)
solution to \eqref{equU} is
\begin{equation}
U_{t}= U_0+\mathrm{e}^{-aG(t+T)} W_{\alpha^{-1}(t)}^{+}. \label{defU}
\end{equation}
By a martingale comparison theorem, we prove that $|\widetilde{Y}_t|
\leq U_t$ a.s. Indeed, let $l$ be a function of class $\mathcal{C}^{2}$
such that $ \forall x > 0\dvtx l(x)> 0$ and $l{'}(x) > 0$, and $\forall x
\leq 0\dvtx l(x)=0.$ We apply It\^{o}'s formula to
$l(|\widetilde{Y}_{t}|-U_{t})$ to show that, on the event $\{
|\widetilde{Y}_{s}| > U_{s}\}$, we have $l(|\widetilde{Y}_{t}|-U_{t})
\leq 0 $ a.s. Finally, as $\alpha^{-1}(t)= \int_0^t
\mathrm{e}^{2aG(s+T)}\,
\mathrm{d}s$, we conclude by the law of the iterated logarithm (LIL),
that a.s. $U_{t}=\mathrm{O}( \sqrt{g(t+T)^{-1} \log G(t+T)})$.

(2) Suppose that $d\ge 2$. Define $\tau := \inf \{t>0;
\widetilde{Y}_t=0\}$. It\^{o}'s formula implies
\[
\mathrm{d}|\widetilde{Y}_{t\wedge \tau}|= \mathrm{d}W_{t\wedge \tau} -
g(t\wedge\tau+T)\biggl( \frac{\widetilde{Y}_{t\wedge
\tau}}{|\widetilde{Y}_{t\wedge \tau}|},\nabla V(\widetilde{Y}_{t\wedge
\tau})\biggr)\,\mathrm{d}t - \frac{|\widetilde{Y}_{t\wedge \tau}|}{r+t\wedge
\tau+T}\,\mathrm{d}t + \frac{d-1}{2|\widetilde{Y}_{t\wedge
\tau}|}\,\mathrm{d}t, \label{equnormeY}
\]
where $W_t=\int_{0}^t
(\frac{\widetilde{Y}_s}{\vphantom{^{''}}|\widetilde{Y}_s|},\mathrm{d}B_s^T )$ is a
standard Brownian motion. The condition $\nabla^2V(0)> 0$ implies that
there exists $a>0$ such that
\begin{equation}
\forall |y| \le \varepsilon\qquad (y,\nabla V(y)) \geq a |y|^2.
\end{equation}
Let us introduce the $(d-1)$-dimensional Bessel process $R$. Consider
the time-shifted process $U_t:=\mathrm{e}^{-aG(t+T)} R_{\int_{0}^t
\mathrm{e}^{2aG(s+T)}\,\mathrm{d}s}$, which is the non-negative (strong) solution
to
\begin{equation}\label{equUbessel}
\mathrm{d}U_t= \mathrm{d}\beta_t^T -ag(t+T)U_t\,\mathrm{d}t +
\frac{d-1}{2U_t}\,\mathrm{d}t,
\end{equation}
where $\beta$ is a Brownian motion. On the event $\{\forall s \geq T;
\vert Y_{s} \vert < \varepsilon \}$, we apply the previous comparison
theorem to obtain a.s. $|\widetilde{Y_t}| \leq U_t$. On the other hand,
$R$ is the radial part of a $d$-dimensional Brownian motion. So the LIL
implies a.s. $R_t =\mathrm{O}(\sqrt{(t+T)\log\log(t+T)})$ and $U_t =
\mathrm{O}(\sqrt{g(t+T)^{-1}\log G(t+T)})$.

Now we prove that $\mathbb{P}(\forall s \geq T; \vert Y_{s}-m \vert <
\varepsilon ) >0$. Let $\tau_T :=\inf \{s>T; |Y_{s} -m|>\varepsilon\}$.
For all $T<t<\tau_T$, we have a.s. $|Y_{t} -m| \leq U_t$. By the LIL
applied to $U$, we conclude that, for $T$ large enough,
$\mathbb{P}(\sup_{s\geq T}U_s <\varepsilon)
>0$ and finally $\mathbb{P}(\tau_T = \infty)>0$.
\end{pf}

\begin{coro}
Suppose that $g(t)^{-1}\log G(t)$ converges to 0 when $t$ tends to
infinity. Then there exists $T_0>0$ such that for all $T>T_0$, the
process $Y_t$ converges almost surely to $m$ on the event $\{\forall s
\geq T;\vert Y_{s}-m \vert < \varepsilon \}$.
\end{coro}

\subsection{Case of an unstable critical point: Local maximum or saddle point}

If $M$ is a local maximum of $V$, then as $ \Delta V(M) < 0$,
$\varepsilon_{1} := \sup \{ \varepsilon ; \forall \vert y \vert <
\varepsilon,  \Delta V(M+y) < 0 \}$ exists and is finite.

If $M$ is a saddle point of $V$, then as $\nabla^2 V$ admits
a negative eigenvalue in $M$, there exists an unstable direction $e$
associated with $M$. Let $P_e \dvtx \mathbb{R}^d \mapsto \mathbb{R}e$ be
the projection on $\mathbb{R}e$. The amount
$\varepsilon_2:=\sup\{\varepsilon ;  \forall |y|<\varepsilon
,\partial_{ee}^2 V(M+y) <0 \mbox{ and }
(\partial_eV(P_e(y)),\partial_eV(y)) > 0 \}$ exists and is finite.

\begin{propo}\label{maxY}
Let $M$ be an unstable critical point of $V$. If $M$ is a local
maximum, suppose that $0< \varepsilon < \varepsilon_1$.
 If $M$ is a saddle point, suppose that $0< \varepsilon < \varepsilon_{2}$.

Let $T$ be a positive stopping time such that $\vert Y_{T}-M\vert <
\varepsilon$. Then
\[
\mathbb{P}( \forall s \geq T;\vert Y_{s} - M\vert
< \varepsilon ) = 0.
\]
\end{propo}
\begin{pf}
Note that $T<\infty$ a.s. by Proposition \ref{propVcritique}. Suppose
that $M$ is a local maximum and $M=0$, because the method of the proof
is the same for $M\neq 0$ (in that case, we have an additional term
$M\log (t+T)$). Let $D(t,Y_t)$ be the drift term of $V(Y_t)$. On the
event $A:=\{ \forall s \geq T;\vert Y_{s}\vert < \varepsilon \}$, we
obtain
\[
D(t+T,Y_{t+T}) = g(t+T) \vert \nabla V(Y_{t+T}) \vert^{2} +
\frac{(Y_{t+T},\nabla V(Y_{t+T}))}{r+t+T} - \frac{1}{2} \Delta
V(Y_{t+T})
 \geq  \frac{C_{1}}{t+T} + C_{2},
\]
where $C_{1}=\frac{1}{2}\inf \{(y,\nabla V(y)); \vert y\vert <
\varepsilon\}$ and $C_{2}= -\frac{1}{2}\sup \{\Delta V(y); \vert y
\vert < \varepsilon\} >0$. We thus find for~$t$ large enough that
$D(t+T,Y_{t+T})\geq C>0$ and so
\begin{equation}\label{equVmax}
\mathbb{E}( V(Y_{t+T})\1_{A}) \leq \mathbb{E}(V(Y_{T})\1_A ) -
Ct\mathbb{P}(A) + \mathrm{o}(t).
\end{equation}
Finally, this last inequality is impossible (since $V$ is positive)
unless $\mathbb{P}(A) =0$.

Suppose now that $M$ is a saddle point. We apply It\^{o}'s formula to
$x\mapsto V(P_e(x))$ and follow the previous computation with $C_1
=\frac{1}{2}\inf \{(P_e(y),\partial_eV(P_e(y))); \vert y\vert
 < \varepsilon\}$ and $C_{2} = -\frac{1}{2} \sup \{\partial_{ee}^2 V(P_e(y));
 \vert y \vert < \varepsilon\} >0$.
\end{pf}

\section{Asymptotics} \label{s:X}
Throughout this section, we always suppose that $g(\infty) = +\infty$
and $g'(t)/g^2(t)$ converges to 0, even if we do not remind the reader
in the statements of the results. In particular, it implies that for
all $T>0$, $G^{-1}(t+T) - G^{-1}(t)$ goes to 0 when $t$ tends to
infinity.
\subsection{Ergodicity}\label{ss:ergoY}

\begin{lemma} \label{lemYborneL2}
The process $Y$ is bounded in $L^{2}$.
\end{lemma}
\begin{pf}
We show a stronger result: $\mathbb{E}V(Y_t)$ is bounded.
 For all $n\in \mathbb{N}$, define the
stopping time $\tau_n=\inf\{t\ge 0; |Y_t|>n\}$. Then there exists $C>0$
such that we have by localization:
\[
\mathbb{E} V(Y_{t\wedge \tau_n} ) \leq
\mathbb{E} V(Y_0) + \mathrm{e}^{Ct} <\infty.
\]
Let $n$ go to infinity and use
Fatou's lemma to find, for all $t\geq 0$, that $V(Y_t) \in L^1$. For
$t$ large enough, we have that $-g(t) V(x) + aV(x) \le -\frac{1}{2}
g(t) V(x)$. So, as $W$ is $c$-strictly convex and by the growth
hypothesis \eqref{growth}, the following holds for $t$ large enough:
\[
\frac{\mathrm{d}}{\mathrm{d}t} \mathbb{E}V(Y_t)
\leq - \frac{1}{2} g(t) \mathbb{E}V(Y_t) .
\]
Now, solving $\dot{u} = -\frac{1}{2} g(t) u$ leads to $\mathbb{E} V(Y_t) = \mathrm{O}(1)$.
\end{pf}
In order to obtain the ergodic result for $Y$, we introduce a dynamical
system $\phi$, whose asymptotics are close to $Y$ (see \cite{coursdea}
for more details):
\begin{defn}
The process $Y$ is an asymptotic pseudotrajectory for the flow $\phi$
if
\begin{equation}\label{tightp}
\forall T>0,\forall \alpha>0\qquad  \lim_{t\rightarrow +\infty}
\mathbb{P}\Bigl( \sup_{0\leq h\leq T} |Y_{t+h} - \phi_{h}(Y_t)| \geq
\alpha\Bigr)= 0.
\end{equation}
\end{defn}

\begin{propo}\label{pta}
Let $\phi\dvtx \mathbb{R}_+\times \mathbb{R}^d\rightarrow \mathbb{R}^d$ be
the flow generated by
\begin{equation}\label{flow}
\frac{\mathrm{d}}{\mathrm{d}t}\phi_t(x) = - \nabla V(\phi_{t}(x));\qquad
\phi_0(x)=x.
\end{equation}
Then $(Y_{G^{-1}(t)},t\ge 0)$ is an asymptotic pseudotrajectory for
$\phi$.
\end{propo}
\begin{pf}
Let $\widetilde{Y}_t=Y_{G^{-1}(t)}$ and
$\widetilde{B}_t=B_{G^{-1}(t)}$. We will use Markov's inequality\vspace*{1pt} and
then prove that $\lim_{t\rightarrow\infty} \mathbb{E}(\sup_{0\leq h\leq
T} |\widetilde{Y}_{t+h} - \phi_h(\widetilde{Y}_{t})| )=0.$

Define $\kappa(t):= (r+G^{-1}(t))g(G^{-1}(t)) $. A simple computation
yields to
\[
\widetilde{Y}_{t+h} - \phi_h(\widetilde{Y}_{t}) =
\widetilde{B}_{t+h}-\widetilde{B}_{t} + \int_{0}^{h} \bigl( \nabla
V(\phi_s(\widetilde{Y}_{t})) - \nabla V (\widetilde{Y}_{t+s}) \bigr)\,
\mathrm{d}s - \int_0^h \widetilde{Y}_{t+s}
\frac{\mathrm{d}s}{\kappa(t+s)}.
\]
Applying It\^{o}'s formula to $h\mapsto \mathrm{e}^{-2\tilde{C}h}\vert
\widetilde{Y}_{t+h} - \phi_h(\widetilde{Y}_{t})\vert^2$, we have:
\begin{eqnarray*}
&&\frac{1}{2} \mathrm{e}^{2\tilde{C}h}\,\mathrm{d}\bigl(\mathrm{e}^{-2\tilde{C}h}\vert
\widetilde{Y}_{t+h} - \phi_h(\widetilde{Y}_{t})\vert^2\bigr)\\
&&\quad = \bigl(\widetilde{Y}_{t+h} - \phi_h(\widetilde{Y}_{t}),
\mathrm{d}\widetilde{B}_{t+h}\bigr) +\frac{(\widetilde{Y}_{t+h} -
\phi_h(\widetilde{Y}_{t}), \widetilde{Y}_{t+h} )}{\kappa(t+h)}\,
\mathrm{d}h \\
&&\qquad{}-\tilde{C}\vert \widetilde{Y}_{t+h}
-\phi_h(\widetilde{Y}_{t})\vert^2\,\mathrm{d}h + \bigl(\widetilde{Y}_{t+h} -
\phi_h(\widetilde{Y}_{t}), \nabla V(\phi_h(\widetilde{Y}_{t}))- \nabla
V (\widetilde{Y}_{t+h}) \bigr)\,\mathrm{d}h\\
&&\qquad{} +
\frac{\mathrm{e}^{2\tilde{C}h}}{2g({G^{-1}(t+h)})}\,\mathrm{d}h.
\end{eqnarray*}
First, we notice that $(G^{-1}(t))'=1/g(G^{-1}(t))$. By the convexity
assumption on $V$, we also remark that
\[
-\tilde{C}\vert
\widetilde{Y}_{t+h} -\phi_h(\widetilde{Y}_{t})\vert^2 +
\bigl(\widetilde{Y}_{t+h} - \phi_h(\widetilde{Y}_{t}), \nabla
V(\phi_h(\widetilde{Y}_{t}))- \nabla V (\widetilde{Y}_{t+h}) \bigr)\le 0,
\]
and so we deduce the following upper bound for all $0\le h\le T$:
\begin{eqnarray}\label{eq:pta}
&&\frac{1}{2} \vert \widetilde{Y}_{t+h} -
\phi_h(\widetilde{Y}_{t})\vert^2\nonumber\\
 &&\quad\leq \mathrm{e}^{2\tilde{C}h}\int_0^h
\mathrm{e}^{-2\tilde{C}s}\bigl(\widetilde{Y}_{t+s} -
\phi_s(\widetilde{Y}_{t}),\mathrm{d}\widetilde{B}_{t+s}\bigr)
+ \frac{\mathrm{e}^{2\tilde{C}T}}{2} \bigl( G^{-1}(t+T) - G^{-1}(t)
\bigr)\\
&&\qquad{}+ \mathrm{e}^{2\tilde{C}h}\int_0^h \bigl(\widetilde{Y}_{t+s} -
\phi_s(\widetilde{Y}_{t}), \widetilde{Y}_{t+s} \bigr)
\frac{\mathrm{d}s}{\kappa(t+s)}.\nonumber
\end{eqnarray}
To conclude, we will now find an upper bound for each right-hand term
of \eqref{eq:pta}. By the Burkholder--Davis--Gundy (BDG) inequality for
the local martingale $\int_0^h \mathrm{e}^{-2\tilde{C}s}(\widetilde{Y}_{t+s} -
\phi_s(\widetilde{Y}_{t}), \mathrm{d} \widetilde{B}_{t+s})$ and a rough
upper bound for its quadratic variation, there exists $\alpha>0$ such
that:
\begin{eqnarray*}
&&\mathbb{E} \biggl( \sup_{0\leq h\leq T} \int_0^h \mathrm{e}^{-2\tilde{C}s}\bigl(\widetilde{Y}_{t+s} -
\phi_s(\widetilde{Y}_{t}),\mathrm{d}\widetilde{B}_{t+s}\bigr) \biggr)\\
&&\quad \le
\alpha\bigl(G^{-1}(t+T) - G^{-1}(t)\bigr) \mathbb{E} \Bigl( \sup_{0\leq h\leq T} \vert
\widetilde{Y}_{t+h} - \phi_h(\widetilde{Y}_{t})\vert^2
\Bigr)^{1/2}.
\end{eqnarray*}
We now estimate the remaining term of \eqref{eq:pta}
by the triangle inequality. As $\kappa$ is non-decreasing, we have the
following bound by Lemma \ref{lemYborneL2}:
\[
\mathbb{E} \int_{0}^T \biggl( \frac{\vert \widetilde{Y}_{t+s}
\vert^{2}}{\kappa(t+s)} +  \frac{\vert \widetilde{Y}_{t+s} - \phi_s
(\widetilde{Y}_{t}) \vert^{2}}{\kappa(t+s)}\biggr)\, \mathrm{d}s  \leq \frac{M
T}{\kappa(t)} + \frac{T}{\kappa(t)}\mathbb{E} \Bigl( \sup_{0\leq h\leq T}
\vert \widetilde{Y}_{t+h} - \phi_h(\widetilde{Y}_{t}) \vert^2 \Bigr).
\]
So we obtain for $t$ large enough:
\[
\mathbb{E} \Bigl( \sup_{0\leq h\leq T}\vert \widetilde{Y}_{t+h} -
\phi_h(\widetilde{Y}_{t})\vert^2 \Bigr) \leq  4\mathrm{e}^{4\tilde{C}T} \bigl(G^{-1}(t+T)
- G^{-1}(t)\bigr) + 4M\mathrm{e}^{2\tilde{C}T}\frac{T}{\kappa(t)},
\]
and the result follows since $G^{-1}(t+T) - G^{-1}(t)$ and
$1/\kappa(t)$ converge to 0.
\end{pf}

\begin{lemma}\label{tightm}
Suppose that for all $T>0$, $G^{-1}(t+T) - G^{-1}(t)$ vanishes when $t$
tends to infinity. Then $(\mu_t^{G^{-1}} :=\frac{1}{t}\int_0^t
\delta_{Y_{G^{-1}(s)}}\,\mathrm{d}s,t\geq 0)$ is a tight family of
measures.
\end{lemma}
\begin{pf}
It is enough to show that a.s. $\varphi(t) := \int_0^t
V(Y_{G^{-1}(s)})\,
\mathrm{d}s = \mathrm{O}(t)$. Let $A>0$ and $K$ be a compact set such that
$\forall x \in K^c$, $V(x) \geq A$. Then $\mu_t^{G^{-1}}(V) \geq A
\mu_t^{G^{-1}}(\1_{K^c})$.
 From the growth assumption \eqref{growth},
there exists $a>0$ and for all $\varepsilon>0$, there exists
$k_{\varepsilon}$ such that $ \Delta V \leq  aV$ and $V \leq
k_\varepsilon + \varepsilon \vert \nabla V\vert^2. $ It then yields
\begin{equation}\label{tight}
\varphi(t) \leq k_\varepsilon t + \varepsilon \int_0^t \bigl| \nabla
V\bigl(Y_{G^{-1}(s)}\bigr)\bigr|^2\,\mathrm{d}s\quad  \mbox{and}\quad  \int_0^t \Delta
V\bigl(Y_{G^{-1}(s)}\bigr)\,\mathrm{d}s \leq a\varphi(t).
\end{equation}
Applying It\^{o}'s formula to $V(Y_{G^{-1}(t)})$, we obtain
\begin{eqnarray}\label{itoZ}
\hspace*{-25pt}V\bigl(Y_{G^{-1}(t)}\bigr) - V(Y_{0}) &=& \int_{0}^{G^{-1}(t)} (\nabla
V(Y_s),\mathrm{d}B_s) - \int_0^t \frac{(Y_{G^{-1}(s)},\nabla
V(Y_{G^{-1}(s)}))}{(r+G^{-1}(s))g(G^{-1}(s))}\,\mathrm{d}s\nonumber\\ [-8pt]\\ [-8pt]
&&{}- \int_0^t \bigl| \nabla V\bigl(Y_{G^{-1}(s)}\bigr)\bigr|^2\,\mathrm{d}s +
\frac{1}{2} \int_0^t \Delta V\bigl(Y_{G^{-1}(s)}\bigr)
\frac{\mathrm{d}s}{g(G^{-1}(s))}.\nonumber
\end{eqnarray}
Consider the (local-)martingale term of \eqref{itoZ}. On the set
$\{\int_{0}^{\infty} |\nabla V(Y_s)|^2\, \mathrm{d}s <\infty\}$, it is
bounded in $L^2$ and thus converges. Whereas on the set
$\{\int_{0}^{\infty} |\nabla V(Y_s)|^2\, \mathrm{d}s =\infty\}$, the
strong LLN implies that, for $t$ large enough, a.s.
\[
\int_{0}^{G^{-1}(t)} (\nabla V(Y_s),\mathrm{d}B_s) \leq
\frac{1}{2}\int_0^t \bigl|\nabla V\bigl(Y_{G^{-1}(s)}\bigr)\bigr|^2\,\mathrm{d}s.
\]
By
\eqref{itoZ}, we find for $t$ large enough
\begin{eqnarray*}
\int_0^t \bigl| \nabla V\bigl(Y_{G^{-1}(s)}\bigr)\bigr|^2\,\mathrm{d}s &\leq &
\int_0^t  \Delta V (Y_{G^{-1}(s)})  \frac{\mathrm{d}s}{g \circ G^{-1}(s)} - 2
V\bigl(Y_{G^{-1}(t)}\bigr) + 2V(Y_{0})\\
&&{}- 2\int_0^t \frac{(Y_{G^{-1}(s)},\nabla V(Y_{G^{-1}(s)}))}{(r+G^{-1}(s))g(G^{-1}(s))}\,\mathrm{d}s\\
&\leq & \frac{a\varphi(t)}{g( G^{-1}(t))} + 2V(Y_{0}) + \mathrm{O}\biggl(\int_0^t
\frac{\mathrm{d}s }{G^{-1}(s) g( G^{-1}(s))} \biggr).
\end{eqnarray*}
So we have a.s. $\int_0^t \vert \nabla V(Y_{G^{-1}(s)})\vert^2\,
\mathrm{d}s= \mathrm{O}(t) + a\varphi(t).$ Putting this result in \eqref{tight}
and choosing $\varepsilon$ small enough, we conclude that $\varphi(t) =
\mathrm{O}(t)$ a.s.
\end{pf}

\begin{theo}\label{th:ergo}
The process $Y$ satisfies the pointwise ergodic theorem. More
precisely, there exist some (deterministic) constants $a_i\geq 0$, such
that $\sum a_i =1$ and $\mu_t^Y = \frac{1}{t}\int_0^t \delta_{Y_s}\,
\mathrm{d}s$ converges (for the weak convergence of measures) toward
$\sum_{1\leq i\leq n}  a_i\delta_{m_i}$.
\end{theo}
\begin{pf}
By Bena\"{i}m and Schreiber \cite{BeSc}, Proposition \ref{pta} implies
that the limit points of the empirical measure of $Y_{G^{-1}(t)}$ are
included in the set of all the ``invariant measures'' for
 $\frac{\mathrm{d}}{\mathrm{d}t}\phi_t(x) = -\nabla
V(\phi_t(x))$ with the initial condition $\phi_0(x) =x$. All these
invariant measures are included in $\operatorname{Span}\{\delta_{m_1}, \ldots
,\delta_{m_n},\delta_{M_1},\ldots ,\delta_{M_p}\}$. Let $\mu_t^{G^{-1}}
= \frac{1}{t} \int_0^t \delta_{Y_{G^{-1}(s)}}\,\mathrm{d}s$. By Lemma
\ref{tightm}, the empirical measure $\mu^{G^{-1}}_t$ converges. One
also shows that $\mu_t$ is a Cauchy sequence in $L^1$: there exists
$C>0$ such that for any $s>0$,
\[
|\mathbb{E}\overline{\mu}_{t+s}
-\mathbb{E}\overline{\mu}_t| \leq \frac{s}{t(t+s)}
\int_0^t\mathbb{E}|X_u|\,\mathrm{d}u + \frac{1}{t+s}\int_t^{t+s}
\mathbb{E}|X_u|\, \mathrm{d}u \leq C \frac{s}{t+s}.
\]
So the limit
measure of $\mu_t^{G^{-1}}$ writes $\sum_{i=1}^n a_i \delta_{m_i} +
\sum_{i=1}^p b_i \delta_{M_i}$ (where $a_i, b_i$ are non-negative
constants such that $\sum (a_i+b_i) = 1$). And the same result holds
for $\mu_t$. Indeed, for all continuous bounded functions $\psi$ and
$t>s$, we have (by an integration by parts)
\begin{eqnarray*}
\int_{s}^{t} \psi(Y_u)\,\mathrm{d}u  &=& \frac{G(t)}{g(t)}
\mu_{G(t)}^{G^{-1}} (\psi) - \frac{G(s)}{g(s)} \mu_{G(s)}^{G^{-1}} (\psi) +
\int_{s}^{t} \frac{g^{\prime}(u)G(u)}{g^{2}(u)}
\mu_{G(u)}^{G^{-1}} (\psi) \,\mathrm{d}u\\
&=& (t-s) \mu_{G(t)}^{G^{-1}} (\psi) + \frac{G(s)}{g(s)}
 \bigl(\mu_{G(t)}^{G^{-1}} (\psi) - \mu_{G(s)}^{G^{-1}} (\psi) \bigr) \\
 &&{}+ \int_{s}^{t} \frac{g^{\prime}(u)G(u)}{g^{2}(u)}
 \bigl(\mu_{G(u)}^{G^{-1}} (\psi) - \mu_{G(t)}^{G^{-1}} (\psi)\bigr)\, \mathrm{d}u.
 \end{eqnarray*}
As $\mu_{G(t)}^{G^{-1}} (\psi)$ converges a.s., we deduce that
\[
\mu_{t} \psi = \mathrm{o}(1) + \mu_{G(t)}^{G^{-1}} (\psi) + \frac{1}{t} \int_{s}^{t} \frac{g^{\prime}(u)G(u)}{g^{2}(u)}
\bigl(\mu_{G(u)}^{G^{-1}} (\psi)  - \mu_{G(t)}^{G^{-1}} (\psi)\bigr)\, \mathrm{d}u.
\]
So, by Lemma \ref{l:ipp}, $\mu_t$ converges. We also wish to show that
$b_i=0$ for all $i$. Proposition \ref{maxY} implies that, for an
unstable critical point $M$, there exists a direction $j$ such that for
all $\varepsilon >0$, $\mathbb{P}(\forall s\ge T, |Y_s^{(j)} -
M^{(j)}|\le \varepsilon) = 0$. Consider a  continuous function $f$,
supported by a small ball (of radius $\alpha>0$) around $M$: $f$
vanishes in all critical points except $M$ and $f(M)=1$. Then we have
a.s. $\int_0^t \1_{\{| Y_s^{(j)}-M^{(j)}|\le \alpha \}}\,\mathrm{d}s =
\mathrm{o}(t)$ and $\frac{1}{t}\int_0^t f(Y_s)\,\mathrm{d}s$ converges almost
surely to $b$. So, we conclude that $b=0$.
\end{pf}

At this stage, we have proved that $Y$ satisfies the pointwise ergodic
theorem. The main question is whether $X$ also satisfies the pointwise
ergodic theorem or not. To answer it, we remind that $\overline{\mu}_t$
converges a.s. if and only if $\int_0^t Y_s \frac{\mathrm{d}s}{r+s}$
converges, and if $Y_t\mathop{\stackrel{\mathrm{a.s.}}{\longrightarrow}}\limits_{t\rightarrow\infty} 0$
polynomially fast, then $\overline{\mu}_t$ converges a.s.

\begin{propo}
The measure $\mu_t$ converges weakly if and only if $\overline{\mu}_t$
converges a.s.
\end{propo}
\begin{pf}
We have shown in Theorem \ref{th:ergo} that $Y$ is pointwise ergodic.
Consider the Fourier transform of $\mu_t$ and recall that $X_t = Y_t +
\overline{\mu}_t$. We have for all $u\in \mathbb{R}^d$:
\[
\frac{1}{t}\int_0^t \mathrm{e}^{\mathrm{i}(u,X_s)}\,\mathrm{d}s =
\frac{\mathrm{e}^{\mathrm{i}(u,\overline{\mu}_\infty)}}{t}\int_0^t
\mathrm{e}^{\mathrm{i}(u,Y_s)}\,
\mathrm{d}s + \frac{1}{t}\int_0^t \mathrm{e}^{\mathrm{i}(u,Y_s)}
\bigl(\mathrm{e}^{\mathrm{i}(u,\overline{\mu}_s)} - \mathrm{e}^{\mathrm{i}(u,\overline{\mu}_\infty)}
\bigr)\,
\mathrm{d}s.
\]
The first right member converges a.s. to
$\mathrm{e}^{\mathrm{i}(u,\overline{\mu}_\infty)}\sum_{1\le p\le n} \mathrm{e}^{\mathrm{i}(u,m_p)}$. For
the second\vspace*{1pt} right member, Ces\`{a}ro asserts that it converges a.s. to 0
if and only if $\overline{\mu}_t$ converges a.s.
\end{pf}

\subsection{Almost sure convergence toward the local minima of $V$}

Let $0<\varepsilon< \varepsilon_0$ and $T>T_{0}$ be as in Section
\ref{s:extrema}. Let $m$ be a local minimum of $V$ such that $\vert
Y_{T} - m \vert < \varepsilon$.

\begin{lemma}
If $\lim_{t\rightarrow\infty}g(t)^{-1}\log G(t) = 0$, then
for all $\alpha>0$ we have $\int_0^\infty \mathrm{e}^{-\alpha g(t)}\,\mathrm{d}t
<+\infty$.
\end{lemma}
\begin{pf}
For all $\varepsilon >0$, there exists $t$ such that, for all $s\geq
t$, we have $g(s) \geq \varepsilon^{-1} \log G(s).$ Moreover, there
exists $a>0$ such that for $t$ large enough $g(t) \geq a$ and then
$G(t) \geq at$. So, we conclude that $\int_1^\infty
\mathrm{e}^{-\varepsilon^{-1}\alpha\log (at)}\,\mathrm{d}t <\infty$ (choose, for
instance, $\varepsilon = \alpha/2$).
\end{pf}

\begin{propo} \label{cvYai}
If $g(t)^{-1} \log G(t)$ converges to 0, then $Y_t$ converges a.s. and
for all $i$, we have $\mathbb{P}(\lim_{t\rightarrow\infty}
Y_t= m_i)>0$ and $\mathbb{P}(\lim_{t\rightarrow\infty}Y_t=
M_i )=0$.
\end{propo}
\begin{pf}
Bena\"{i}m (\cite{coursdea}, Proposition 4.6) asserts that if $-\nabla
V(x)$ is a continuous globally integrable vector field, and if for all
$\alpha>0$ we have $\int_0^\infty \mathrm{e}^{-\alpha g\circ
G^{-1}(t)}\,
\mathrm{d}t <+\infty$ and $\mathbb{P}(\sup_t |Y_t|<\infty) =1$, then
$Y$ is a.s. an asymptotic pseudotrajectory for the flow induced by
$-\nabla V$. Actually, the first and last conditions are fulfilled
here. Moreover, as $G^{-1}$ is non-decreasing, the (finite) integral
$\int_0^\infty \mathrm{e}^{-\alpha g(t)}\,\mathrm{d}t$ is an upper bound for the
preceding one. Thus, $Y$ is a.s. an asymptotic pseudotrajectory for the
flow $\phi$ defined by \eqref{flow} and $\phi$ restricted to the limit
points of $Y$ does not admit any other attractor than the set of limit
points. Finally, $Y$ converges a.s. and its limit points are included
into the set $\{x; \nabla V(x) =0\}$.

If $Y$ converges to $Y_\infty$, then, due to Proposition \ref{maxY},
the limit $Y_\infty$ is not a local maximum. On the event $\{\forall
s\geq T; |Y_s-m_i|<\varepsilon\}$, occurring with a positive
probability by Proposition \ref{propcvYai}, we have a.s. $|Y_{t+T}
-m_i| \leq U_t$. As $\overline{\lim}_{t\rightarrow \infty}
U_t \sqrt{\frac{g(t)}{\log G(t)}} =1$ a.s., we conclude that $U_{t}
\mathop{\stackrel{\mathrm{a.s.}}{\longrightarrow}}\limits_{t\rightarrow\infty}0.$
\end{pf}

\begin{coro}
Suppose that $\lim g(t)^{-1} \log G(t) =0$.
\begin{enumerate}[(2)]
\item[(1)] If $V$ is a strictly uniformly convex function ($m$ its unique
minimum), then $Y_{t} \mathop{\stackrel{a.s.}{\longrightarrow}}\limits_{t\rightarrow\infty}m.$

\item[(2)] A necessary condition for the almost sure convergence of $Y$ to 0 is
that the potential $V$ admits a unique minimum at 0 (e.g., $V$ is
symmetric and strictly convex).
\end{enumerate}
\end{coro}

\subsection{Back to $X$}\label{ss:psX}

\begin{theo}
Assume that $g(t)^{-1}\log G(t) = \mathrm{O}(h(t)^{-2})$, with
$h\dvtx\mathbb{R}_+\rightarrow \mathbb{R}_+$ such that\break $\int_0^\infty
\frac{\mathrm{d}s}{(1+s)h(s)} < +\infty$. Then, on the set
$\{Y_\infty\neq 0\}$, $\frac{X_t}{\log t}$ converges to $Y_\infty$.
Moreover:
\begin{enumerate}[(2)]
    \item[(1)] If 0 is the unique local minimum of $V$, then
    \[
    \mathbb{P}\biggl(\lim_{t\rightarrow \infty}X_t= \overline{\mu}+ \int_0^\infty Y_s
    \frac{\mathrm{d}s}{r+s} \biggr) = \mathbb{P} \biggl( \lim_{t\rightarrow \infty} \overline{\mu}_t
    = \overline{\mu} + \int_0^\infty Y_s \frac{\mathrm{d}s}{r+s} \biggr) = 1;
    \]
    \item[(2)] If $V$ admits 0 as a
     (non-unique) local minimum, then $X_t$ converges
      on the event $\{Y_t \stackrel{a.s.}{\longrightarrow} 0\}$ and
       diverges elsewhere. More precisely, one has
     \[
       \mathbb{P}\biggl(\lim_{t\rightarrow \infty} X_t= \overline{\mu}
       + \int_0^\infty Y_s \frac{\mathrm{d}s}{r+s} \biggr)
       + \mathbb{P}\Bigl(\lim_{t\rightarrow \infty}|X_t| = \infty\Bigr) =
       1
       \]
    and
      \[
       1>\mathbb{P}\biggl(\lim_{t\rightarrow \infty}X_t
       = \overline{\mu}+ \int_0^\infty Y_s \frac{\mathrm{d}s}{r+s}\biggr)
       =\mathbb{P}\biggl(\lim_{t\rightarrow \infty} \overline{\mu}_t = \overline{\mu}
       + \int_0^\infty Y_s \frac{\mathrm{d}s}{r+s}\biggr)>0;
     \]
    \item[(3)] If 0 is not a local minimum of
    $V$, then
   \[
    \mathbb{P}\Bigl(\lim_{t\rightarrow \infty}|X_t|= \infty \Bigr)
    = \mathbb{P}\Bigl(\lim_{t\rightarrow \infty}|\overline{\mu}_t|= \infty\Bigr)= 1.
    \]
\end{enumerate}
\end{theo}
\begin{pf}
Denote $m=(m^{(1)},\ldots, m^{(d)})$. First, suppose that $m=0$. By
Proposition \ref{cvYai}, $Y_t$ converges toward 0 with a positive
probability. On this event, Proposition \ref{propcvYai} implies that
the integral $\int_0^t \frac{Y_s}{r+s}\,\mathrm{d}s$ converges. So,
$\overline{\mu}_t $ converges toward this (limit) integral and the
result follows for $X$. On the other hand, if $m\neq 0$, then
$\mathbb{P}(Y_t \rightarrow m) >0$ and so the $j$th coordinate of
$\overline{\mu}_t$ converges to  $\operatorname{sgn}(m^{(j)}) \infty$. So the
direction $j$ is unstable and $X_t$ does not converge a.s. Moreover, on
the set $\{Y_\infty \neq 0\}$, we have
\[
\biggl|\frac{\overline{\mu}_t}{\log t}- Y_\infty \biggr| \leq
\frac{1}{\log t}\int_0^t |Y_s -Y_\infty| \frac{\mathrm{d}s}{r+s}\le
\frac{1}{\log t}\int_0^t \sqrt{\frac{\log G(s)}{g(s)}}
\frac{\mathrm{d}s}{r+s}.
\]
The latter upper bound tends to 0 by the law
of the iterated logarithm (Proposition \ref{propcvYai}). As
$\frac{X_t}{\log t} = \frac{\overline{\mu}_t}{\log t} + \frac{Y_t}{\log
t}$, the result follows.
\end{pf}
\begin{rk}
Any polynomial $h$ satisfies the required condition. In particular, one
can choose $g(t) = t^\alpha (\log (1+t))^\beta$ with $\alpha >0$, or
$\alpha = 0$ and $\beta>2$.
\end{rk}

\section*{Acknowledgements} The authors thank O. Raimond and M.
Bena\"{i}m for careful reading and discussions. We are very grateful to
two anonymous referees for their helpful comments. A. Kurtzmann has been
partially supported by the Swiss National Science Foundation grant
200020-112316/1.

\printhistory


\begin{thebibliography}{17}

\bibitem{coursdea}
Bena\"{i}m, M. (1999). Dynamics of stochastic approximation algorithms.
\textit{In S\'{e}m. Prob. XXXIII, Lecture Notes in Math.} \textbf{1709}
1--68. Berlin: Springer.
\MR{1767993}

\bibitem{BeLeRa}
Bena\"{i}m, M., Ledoux, M. and
Raimond, O. (2002). Self-interacting diffusions.
\textit{Probab. Theory  Related Fields} \textbf{122} 1--41.
\MR{1883716}

\bibitem{BeRa05}
Bena\"{i}m, M. and Raimond, O. (2005).
Self-interacting diffusions III: Symmetric
interactions. \textit{Ann. Probab.} \textbf{33} 1716--1759.
\MR{2165577}

\bibitem{BeSc}
Bena\"{i}m, M. and Schreiber, S.J. (2000). Weak asymptotic pseudotrajectories for
semiflows: Ergodic properties.
 \textit{J. Dynam. Differential Equations} \textbf{12} 579--598.
\MR{1800134}

\bibitem{CLeJ}
Cranston, M. and Le Jan, Y. (1995). Self-attracting
diffusions: Two cases studies. \textit{Math. Ann.} \textbf{303} 87--93.
\MR{1348356}

\bibitem{CrMo}
Cranston, M. and Mountford, T.S. (1996).
The strong law of large numbers for a Brownian polymer. \textit{Ann.
Probab.} \textbf{24} 1300--1323.
\MR{1411496}

\bibitem{DuRo}
Durrett, R.T. and Rogers, L.C.G. (1992).
Asymptotic behaviour of Brownian polymers. \textit{Probab. Theory
Related Fields} \textbf{92} 337--349.
\MR{1165516}

\bibitem{HaSh}
Harrison, J.M. and Shepp, L.A. (1981).
On the skew Brownian motion.
\textit{Ann. Probab.} \textbf{9} 309--313.
\MR{0606993}

\bibitem{HeRo}
Herrmann, S. and Roynette, B. (2003). Boundedness and convergence
of some self-interacting diffusions. \textit{Math. Ann.}
\textbf{325} 81--96.
\MR{1957265}

\bibitem{AK}
Kurtzmann, A. (2010). The ODE method for some self-interacting diffusions on $\mathbb{R}^d$.
\textit{Ann. Inst. H. Poincar\'{e} Probab. Statist.} \textbf{46} 618--643.
\MR{2682260}

\bibitem{Lep}
L\'{e}pingle, D. (1978). Sur le comportement aymptotique des martingales
locales.
\textit{S\'{e}m. Probab. (Strasbourg)} \textbf{649} 148--161.
\MR{0520004}

\bibitem{MoTa}
Mountford, T.S. and Tarr\`{e}s, P. (2008).
An asymptotic result for Brownian polymers. \textit{Ann. Inst. H.
Poincar\'{e} Probab. Statist.} \textbf{44}  29--46.
\MR{2451570}

\bibitem{NoRoWi}
Norris, J.R., Rogers, L.C.G. and Williams, D. (1987).
Self-avoiding random walk: A Brownian motion model with local time
drift. \textit{Probab. Theory Related Fields} \textbf{74} 271--287.
\MR{0871255}

\bibitem{Pemantle}
Pemantle, R. (2007).
A survey of random processes with reinforcement.
\textit{Probab. Surveys} \textbf{4} 1--79.
\MR{2282181}

\bibitem{Rai97}
Raimond, O. (1997). Self-attracting diffusions:
Case of constant interaction. \textit{Probab. Theory Related Fields}
\textbf{107} 177--196.
\MR{1431218}

\bibitem{Rai06}
Raimond, O. (2009). Self-interacting diffusions:
A simulated annealing version. \textit{Probab. Theory Related Fields}
\textbf{144} 247--279.
\MR{2480791}

\bibitem{ReY}
Revuz, D. and Yor, M. (1999). \textit{Continuous Martingales and Brownian Motion. Grundlehren der Mathematischen Wissenschaften}, 3rd ed. Berlin: Springer.
\end{thebibliography}
\end{document}